\documentclass[a4paper,12pt]{article} 
\usepackage{amsmath,fancyhdr,amssymb,amsthm, geometry, enumerate, graphicx ,amsfonts,hyperref,mathrsfs}
\title{Dynamics of weakly mixing non-autonomous systems}
\author{Mohammad Salman$^a$, \  Ruchi Das$^{a,\dag}$}
\title{Dynamics of weakly mixing non-autonomous systems}
\theoremstyle{definition}
\newtheorem{defn}{Definition}[section]
\providecommand{\keywords}[1]{\textbf{Keywords :} #1}
\providecommand{\msc}[1]{\textbf{Mathematics Subject Classification(2010)} #1}
\theoremstyle{plain}
\newtheorem{thm}{Theorem}[section]

\newtheorem{cor}[defn]{Corollary}
\newtheorem{lem}[thm]{Lemma}
\theoremstyle{definition}
\newtheorem{exm}{Example}[section]
\newtheorem{rmk}{Remark}[section]
\begin{document}
\date{}
\maketitle

\begin{abstract}
 For a commutative non-autonomous dynamical system we show that topological transitivity of the non-autonomous system induced on probability measures (hyperspaces) is equivalent to the weak mixing of the induced systems. Several counter examples are given for the results which are true in autonomous but need not be true in non-autonomous systems. Wherever possible sufficient conditions are obtained for the results to hold true. For a commutative periodic non-autonomous system on intervals,  it is proved that weakly mixing implies Devaney chaos. Given a periodic non-autonomous system, it is shown that sensitivity is equivalent to some stronger forms of sensitivity on a closed unit interval. 
\end{abstract}


\msc{Primary 54H20; Secondary 37B55}
\bigskip\renewcommand{\thefootnote}{\fnsymbol{footnote}}
\footnotetext{\hspace*{-5mm}
\renewcommand{\arraystretch}{1}
\begin{tabular}{@{}r@{}p{11cm}@{}}
$^\dag$& the corresponding author. \emph{Email addresses}: rdasmsu@gmail.com (R. Das), salman25july@gmail.com (M. Salman)\\
$^a$&Department of Mathematics, University of Delhi, Delhi-110007, India
\end{tabular}}
\vspace{-2mm}
\section{Introduction}
The theory of dynamical systems is a successful mathematical tool to describe time-varying phenomena. Its wide application area varies from simple motion of a pendulum to complicated climate models in physics and complex signal transduction process in biological cells. One of the most significant constituent of dynamical systems theory is chaos which is closely related to different forms of mixing. Parameters in real-world problems are rarely independent of time and thus non-autonomous systems, whose law of behavior is influenced by external forces, are widely useful over the last decades. The external influence can be of different nature, it could be a periodic force or a noisy process which could modulate the functional relationships that define the interactions among systems. Such systems constitute the core mechanism for non-stationary output signals. Also, for an appropriate time-dependence of the external  factor, there is a possibility that a non-autonomous dynamics is stationary. Therefore,  a non-autonomous system acts as a functional generator for both stationary and non-stationary dynamics. In this paper, we consider the following non-autonomous discrete dynamical system, which can be seen as the discrete analogue of a non-autonomous differential equation
\begin{equation}\label{eq} x_{n+1} = f_n(x_n) ,  \  n \geq 1, \end{equation}
where $(X,d)$ is a metric space and $f_n: X \rightarrow X$ is a continuous map, for each $n\geq 1$ and the $n$-th iterate is given by composition of different maps $f_i$'s. When $f_n = f$, for each $n\geq 1$, then the above system becomes autonomous dynamical system $(X, f)$. In 1996, non-autonomous discrete dynamical systems were introduced by Kolyada and Snoha \cite{8}. There are different types of non-autonomous systems such as uniformly convergent, finitely generated, periodic, etc. It often happens that the external factor responsible for the non-autonomicity is itself deterministic, for instance, periodic, which shows the significance of finitely generated non-autonomous systems. The study of complexity and chaos of non-autonomous systems has seen an increasing interest of many researchers  in recent years \cite{1, 5, 11, 14, 16, 17, 18, 19, 21}.

A dynamical system induces two natural systems; one on hyperspaces and the other on probability measures spaces. Such systems are crucial as most of the natural phenomena occur collectively as union of several components and hence non-autonomous system on induced spaces plays an improtant role in dynamics of nature. Probability measures have varied applications in different fields such as physics, finance, biology, etc. In 1975, Bauer and Sigmund initiated study of connection of various topological  properties of dynamical systems with the induced systems in \cite{3}. Since then both the induced systems have been studied by many authors, cf., \cite{2, 4, 9, X, 10, 20}. Recently, we have studied different forms of specification properties on induced systems \cite{15}.   Weak mixing is vital in the study of chaos theory as it implies sensitivity which is a key ingredient of chaos. Also, it is important to measure how much chaotic the system is, which resulted in the study of stronger forms of sensitivity initiated by Moothathu in \cite{12}. Devaney chaos and sensitivity for non-autonomous dynamical systems were introduced by Tian and Chen \cite{18}. Ever since the study of non-autonomous systems started two natural questions regarding Devaney chaos are : When does topological transitivity and density of periodic points imply sensitivity? When does topological transitivity on intervals implies Devaney chaos?  Many authors have given  different sufficient conditions in \cite{11, 21}, which answers the first question. Recently, we proved that on finitely generated non-autonomous systems transitivity and density of periodic points imply thickly syndetical sensitivity and hence sensitivity \cite{14}. The second one is still an open problem.

This paper is organized as follows. In Section 2, we give the preliminaries required for the remaining sections  of the paper. In Section 3, we study the interrelation of weakly mixing and other related properties of a non-autonomous system and its induced probability measures non-autonomous system.  Equivalence of weakly mixing and topological transitivity of a commutative induced non-autonomous system on probability measures space and on hyperspaces is shown. In Section 4, we give counter examples for the results which are true in autonomous systems but not in non-autonomous systems and where ever possible sufficient conditions are obtained for the results to hold true. For a commutative finitely generated non-autonomous system it is shown that every weakly mixing system is thickly sensitive. It is proved that on commutative periodic non-autonomous systems weakly mixing implies Devaney chaos and if an induced autonomous system from a periodic non-autonomous system is topologically transitive on intervals, then the non-autonomous system is Devaney chaotic.  In Section 5, we first give examples to show that in non-autonomous systems sensitivity need not be equivalent to cofinite sensitivity and other stronger forms of sensitivity. Sufficient condition is given under which sensitivity is equivalent to some stronger forms of sensitivity in a non-autonomous system.

\section{Preliminaries}
In this section, we give some definitions and relevant concepts which are required for remaining sections of the paper. Throughout this paper, $(X, d)$ will denote a compact metric space, $\mathbb{N}$  the set of natural numbers and  $B_d(x, \epsilon)$ the open ball of radius $\epsilon>0$ and  center $x$.
 Denote $f_{1,\infty} := \{f_n\}_{n=1}^{\infty}$, and for all positive integers $i$ and $n$, $f_n^i := f_{n+i-1}\circ\cdots\circ f_n$, $f_n^0 := id$ and  the $k^{th}$ iterate by $f_1^{[k]} = \{f_{k(n-1)+1}^k\}_{n=1}^{\infty}$, for any $k\in\mathbb{N}$. 
 
  We say that $(X, f_{1,\infty})$ is a $periodic$ discrete system, if there exists $k\in\mathbb{N}$ such that $f_{n+km}(x) = f_n(x)$, for any $x\in X$, $m\in\mathbb{N}$ and $1\leq n\leq k$. A non-autonomous system $(X, f_{1,\infty})$ is said  to be \emph{finitely generated}, if there exists a finite set $F$ of continuous self maps on $X$ such that each $f_i$ of $f_{1,\infty}$ belongs to $F$. Clearly, every periodic non-autonomous system is finitely generated. For the system (\ref{eq}), the $orbit$ of any point $x\in X$  is the set, $\{f_1^n(x) : n\geq 0\} = \mathcal{O}_{f_{1,\infty}}(x)$. A point $x\in X$ is said to be \emph{periodic}, for the non-autonomous system $(X, f_{1,\infty})$, if there exists $n\in\mathbb{N}$ such that $f_1^{nk}(x) = x$, for every $k\in\mathbb{N}$ \cite{11}. We say that the system $(X, f_{1,\infty})$ has \emph{dense small periodic set} if for any non-empty open subset $U$ of $X$, there exist a closed subset $F\subseteq U$ and $n\in\mathbb{N}$ with $f_1^{nk}(F)\subseteq F$, for every $k\geq 1$. If $V$ is a non-empty, closed and invariant subset of $X$, and no proper subset of $V$ is non-empty, closed and invariant, then $V$ is said to be a \emph{minimal} subset of $(X, f_{1,\infty})$. For a non-autonomous system $(X,f_{1,\infty})$, we put $X^{2}$ = $X \times X$ and $(f_{1,\infty})^{2}$ = ($g_1$, $g_2$, \ldots  , $g_n$, \ldots), where $g_{n}$ = $f_{n}$ $\times$ $f_{n}$, for each positive integer $n$. Therefore, $(X^{2},(f_{1,\infty})^{2})$ is a non-autonomous dynamical system. We have, 
$g_1^n = g_n \circ g_{n-1}\circ\cdots\circ g_2\circ g_1
               = (f_n \times f_n)\circ (f_{n-1} \times f_{n-1})\circ\cdots\circ (f_2 \times f_2) \circ (f_1 \times f_1) = f_{1}^{n}\times f_{1}^{n}$. Similarly we can define $(X^{m},(f_{1,\infty})^{m})$ in general for any positive integer $m$.

Let $\cal{K}$($X)$ denote the hyperspace of all non-empty compact subsets of $X$ endowed with the \textit{Vietoris Topology}. A basis  for Vietoris topology is given by the sets, $\langle U_1, U_2$, $\ldots , U_k\rangle$ = $\{K \in$ $\cal{K}$($X)$: $K \subseteq \bigcup_{i=1}^{k} U_{i}$ and $K\cap U_{i}$  $\ne \emptyset$, for each $i$ $\in \{1, 2, \ldots ,k\}$\}, where $U_1, U_2, \ldots ,U_k$ are non-empty open subsets of $X$. Let $x \in X$, $A \in$ $\cal{K}$($X)$, then we write $N(A, \epsilon)$ = $\bigcup_{a \in A} B_d(\epsilon, a)$. The Hausdorff metric in $\cal{K}$($X)$  induced by $d$, denoted by $\cal{H}$ is defined  by $\cal{H}$($A, B) = \inf \{ \epsilon > 0: A \subseteq N(B, \epsilon) \  \text{and} \  B \subseteq N(A, \epsilon)\}$,  where $A$, $B \in$ $\cal{K}$($X)$. The topology induced by the Hausdorff metric on $\mathcal{K}(X)$ coincides with the Vietoris topology if and only if the space $X$ is compact \cite{7}. Let $(X, f_{1,\infty})$ be a non-autonomous  dynamical system and $\overline{f}_n$ the  continuous function on $\cal{K}$($X)$ induced by $f_n$ defined by $\overline{f}_n(K) = f_n(K)$, $K\in\mathcal{K}(X)$, for each $n\in\mathbb{N}$. Then the sequence $\overline{f}_{1,\infty}$ = ($\overline{f}_1,\ldots, \overline{f}_n, \ldots )$ induces a non-autonomous discrete dynamical system ($\cal{K}$($X), \overline{f}_{1,\infty})$, where $\overline{f}_1^n = \overline{f}_n \circ \cdots \circ \overline{f}_2\circ \overline{f}_1$.

Let $\mathcal{B}(X)$ be the $\sigma$-algebra of Borel subsets of $X$ and $\mathcal{M}(X)$ be the set of all \emph{Borel probability measures} on $(X, \mathcal{B}(X))$ and $\mathcal{M}(X)$ be equipped with the \emph{Prohorov metric $\mathcal{D}$} defined by $\mathcal{D}(\mu, \nu)$ = $\inf\{\epsilon: \mu(A)\leq \nu(N(A, \epsilon))+\epsilon \ \text{and} \ \nu(A)\leq \mu(N(A, \epsilon))+\epsilon$, for each $A\in\mathcal{B}(X)\}$. It is known that topology induced by $\mathcal{D}$ is the \emph{weak*-topology} on $\mathcal{M}(X)$ \cite{6}. For $x\in X$, $\delta_x\in \mathcal{M}(X)$ denote \emph{Dirac point measure}, given by $\delta_x(A) = 1$, if $x\in A$ and $0$ otherwise. Let $\mathcal{M}_n(X) =\{(\sum_{i=1}^n\delta_{x_i})/n : x_i\in X \ \text{(not necessarily distinct)}\}$ and $\mathcal{M}_{\infty}(X) = \bigcup_{n\in\mathbb{N}}\mathcal{M}_n(X)$. It is known that $\mathcal{M}_{\infty}(X)$ is dense in $\mathcal{M}(X)$ and each $\mathcal{M}_n(X)$ is closed in $\mathcal{M}(X)$ \cite{3}. For a non-autonomous system $(X, f_{1,\infty})$, we consider the non-autonomous induced system $(\mathcal{M}(X), \widetilde{f}_{1,\infty})$, where each $\widetilde{f}_{i}: \mathcal{M}(X)\to \mathcal{M}(X)$ is induced continuous function and $\widetilde{f}_{1}^n(\mu)(A) = \mu(f_1^{-n}(A))$, $\mu\in \mathcal{M}(X)$, $A\in\mathcal{B}(X)$ and $f_1^{-n} = (f_1^n)^{-1}$.

\begin{defn}\cite{18} A non-autonomous system $(X, f_{1,\infty})$ is said to be \emph{topologically transitive}, if for each pair of non-empty open subsets $U$, $V$ of $X$, there exists $n\in\mathbb{N}$ such that $f_1^n(U)\cap V\ne\emptyset$. For any two non-empty  open subsets $U$ and $V$ of $X$ denote, $N_{f_{1,\infty}}(U,V)= \{n\in \mathbb{N} : f_1^n(U)\cap V\ne \emptyset \}$. A non-autonomous system $(X, f_{1,\infty})$ is  \emph{totally transitive} if $f_{1,\infty}^{[n]}$ is topologically transitive, for every $n\geq 1$ and  \emph{topologically mixing} if there exists $n\in\mathbb{N}$ such that $N_{f_{1,\infty}}(U,V) \supseteq [n, \infty)$, for any pair of non-empty open subsets $U,$ $V$ of $X$ \cite{5}.
\end{defn}
Recall that a non-autonomous system $(X, f_{1,\infty})$ is said to be \emph{point transitive} if there exists $x\in X$ having dense orbit in $X$.  Also, if $(X, f_{1,\infty})$ is topologically transitive, then it is point transitive and the set of such points is dense in $X$ \cite{16}.
\begin{defn}\cite{16} A non-autonomous dynamical system $(X,f_{1,\infty})$ satisfies \emph{Banks's condition} if for any three non-empty open subsets $U$, $V$ and $W$ of $X$, there exists a positive integer $n$ such that $f_1^n(U)\cap V\ne\emptyset$ and $f_1^n(U)\cap W\ne\emptyset$. 
\end{defn}
\begin{defn} A non-autonomous dynamical system $(X,f_{1,\infty})$ is said to be \textit{weakly mixing of order m} ($m\geq2)$, if for any non-empty open subsets $U_1$, $U_2$, $\ldots$ , $U_m$, $V_1$, $V_2$, $\ldots$ , $V_m$ there exists  $n\in\mathbb{N}$, such that $f_{1}^{n}(U_{i}) \cap V_{i} \ne \emptyset$,  for each $1 \leq i \leq m$. If the non-autonomous system $(X, f_{1,\infty})$ satisfies above condition for $m =2$, then it is called \emph{weakly mixing}.
\end{defn}
Clearly, we have weakly mixing $\implies$ Banks's condition $\implies$ topological transitivity. In fact, for autonomous systems Banks proved that weakly mixing and Banks's condition are equivalent. However, for non-autonomous system this is not true as shown by authors in \cite{16}.
\begin{defn}\cite{18} The system $(X, f_{1,\infty})$ is said to exhibit \emph{sensitive dependence on initial conditions} if there exists $\delta>0$ such that, for every $x\in X$ and any neighborhood $U$ of $x$, there exist $y\in U$ and $n\in\mathbb{N}$ with $d(f_1^n(x), f_1^n(y))>\delta$; $\delta>0$ is called a constant of sensitivity.

 We shall denote $N_{f_{1,\infty}}(U,\delta) = \{n\in \mathbb{N}:  \text{there exist} \ x,y \in U \ \text{such that}\ d(f_1^n(x),f_1^n(y))$ $> \delta \}$, for any arbitrary open subset $U$ of $X$. \end{defn}
A non-autonomous system $(X, f_{1,\infty})$ is said to be \emph{ Devaney chaotic} on $X$ if  it is topologically transitive, has a dense set of periodic points and has sensitive dependence on $X.$
\begin{defn}\cite{5} A non-autonomous system $(X, f_{1,\infty})$ is \emph{cofinitely sensitive} if there exists $\delta>0$ such that for any open subset $U$ of $X$,  $N_{f_{1,\infty}}(U, \delta)$ is cofinite, that is, there exists $N\in\mathbb{N}$ with $N_{f_{1,\infty}}(U, \delta)\supseteq [N, \infty)\cap \mathbb{N}$.
\end{defn}
\begin{defn} A non-autonomous system $(X, f_{1,\infty})$ is \emph{multi-sensitive}, if there exists $\delta>0$ such that for any $m\in\mathbb{N}$ and any collection of non-empty open subsets $V_1$, $V_2$, \ldots, $V_m$ of $X$,  $\bigcap_{i=1}^mN_{f_{1,\infty}}(V_i, \delta)\ne\emptyset$, where $\delta>0$ is  a constant of sensitivity. 
\end{defn}
\begin{defn} A set $F\subseteq\mathbb{N}$ is called $\textit{syndetic}$ if there exists a positive integer $a$ such that $\{i,i+1, \ldots, i+a\}\cap F\ne\emptyset$, for each $i\in\mathbb{N}$. A non-autonomous system $(X, f_{1,\infty})$ is \emph{syndetically sensitive}  if there exists $\delta>0$ such that for any open subset $U$ of $X$,  $N_{f_{1,\infty}}(U, \delta)$ is syndetic.
\end{defn}
\begin{defn}
A $\textit{thick set}$ is a set of integers that contains arbitrarily long runs of positive integers, that is, given a thick set $T$, for every $p\in \mathbb{N}$, there is some $n\in \mathbb{N}$  such that $\{n,n+1,n+2, \ldots, n+p\}$ $\subseteq T$. A non-autonomous system $(X, f_{1,\infty})$ is \emph{thickly sensitive},  if there exists $\delta>0$ such that for any open subset $U$ of $X$, $N_{f_{1,\infty}}(U, \delta)$ is thick.
\end{defn}
\begin{defn} A non-autonomous system $(X, f_{1,\infty})$ is said to be \emph{collectively sensitive}, if for $x_1$, $x_2$, \ldots, $x_m\in X$ and  any $\epsilon>0$ there exists  $y_i\in X$, for each $i\in\{ 1, 2,\ldots, m\}$ with $d(x_i, y_i)<\epsilon$, $n\in\mathbb{N}$ and $i_0$ with $1\leq i_0\leq m$ such that $d(f_1^n(x_i), f_1^n(y_{i_0}))>\delta$ or $d(f_1^n(y_i), f_1^n(x_{i_0}))>\delta$, for some $\delta>0$.
\end{defn}
\begin{defn} Let $S\subseteq\mathbb{N}$ and $|S|$ denote the cardinality of $S$. Then
 $\overline{d}(S) = \limsup\limits_{n\to \infty}\frac{1}{n}|S\cap \{0,1,2, \ldots, n-1\}|$ is called the \textit{upper density} of $S$. A non-autonomous system $(X, f_{1,\infty})$ is said to be \emph{ergodic sensitive}, if there exists $\delta>0$ such that for any open subset $U$ of $X$,  $N_{f_{1,\infty}}(U, \delta)$ has positive upper density.
\end{defn}

\section{Weakly Mixing on Induced Systems}
In  this section, we study the interrelation of weakly mixing and other related properties of $(X, f_{1,\infty})$ and its induced probability measures non-autonomous system $(\mathcal{M}(X), \widetilde{f}_{1,\infty})$. Equivalence of weakly mixing and topological transitivity of a commutative induced non-autonomous system on probability measures space and hyperspaces is proved.

\begin{thm}\label{3.4} If the induced system $(\mathcal{M}(X), \widetilde{f}_{1,\infty})$ is topologically transitive, then  $(X, f_{1,\infty})$ satisfies Banks's condition.
\end{thm}
\begin{proof} Let $U$, $V_1$, $V_2$ be non-empty open subsets of $X$. Let $W_1 = \{\mu\in\mathcal{M}(X) : \mu(U)> 4/5\}$ and $W_2 = \{\mu\in \mathcal{M}(X) : \mu(V_1)>4/5 \ \text{and} \ \mu(V_2)>4/5\}$, then $W_1$ and $W_2$ are non-empty open subsets of $\mathcal{M}(X)$. If $\mu_n\to \mu$ such that $\mu_n\in\mathcal{M}(X)\setminus W_1$, then $\mu_n(U)\leq 4/5$ implying that $\mu(U)\leq 4/5$ and hence $\mu\in\mathcal{M}(X)\setminus W_1$. Thus, $W_1$ is open in $\mathcal{M}(X)$ and similarly $W_2$ is also open in $\mathcal{M}(X)$. Now, since the system $(\mathcal{M}(X), \widetilde{f}_{1,\infty})$ is topologically transitive, therefore there exists $k\in\mathbb{N}$ such that $\widetilde{f}_1^k(W_1)\cap W_2\ne\emptyset$. This implies that there exists $\nu\in W_1$ with $\widetilde{f}_1^k(\nu)\in W_2$ and hence $\widetilde{f}_1^k(\nu)(V_i) = \nu(f_1^{-k}(V_i))>4/5$, for $i=1, 2$. Also, as $\nu(U)>4/5$, so $f_1^k(U)\cap V_1\ne\emptyset$ and $f_1^k(U)\cap V_2\ne\emptyset$. Thus, $(X, f_{1,\infty})$ satisfies Banks's condition.
\end{proof}
Note that if the system satisfies Banks's condition, then it is topologically transitive. Consequently, we have the following result.
\begin{cor}\label{c3.1} If $(\mathcal{M}(X), \widetilde{f}_{1,\infty})$ is topologically transitive, then so is $(X, f_{1,\infty})$.
\end{cor}
\begin{rmk} Converse of Corollary \ref{c3.1} is not true in general as justified by Example \ref{e4.4}
\end{rmk}
\begin{thm}\label{3.5} A non-autonomous system $(X, f_{1,\infty})$ is weakly mixing of all orders if and only if the induced system $(\mathcal{M}(X), \widetilde{f}_{1,\infty})$ is weakly mixing of all orders.
\end{thm} 
\begin{proof} First suppose that $(X, f_{1,\infty})$ is weakly mixing of all orders. Let $\mathcal{U}_1$, \ldots, $\mathcal{U}_k$; $\mathcal{V}_1$, \ldots, $\mathcal{V}_k$ be the collection of non-empty open subsets of $\mathcal{M}(X)$. By density of $\mathcal{M}_\infty(X)$ in $\mathcal{M}(X)$, we can choose $r\in\mathbb{N}$ such that \begin{equation}\label{M1} \mu_j = \frac{1}{r}\sum_{i=1}^r\delta_{x_i^j}\in\mathcal{U}_j \ \text{and} \ \nu_j = \frac{1}{r}\sum_{i=1}^r\delta_{y_i^j}\in\mathcal{V}_j, \ \text{for} \ j=1, 2, \ldots, k.\end{equation} Now, since $(X, f_{1,\infty})$ is weakly mixing of all orders, therefore for open neighborhoods $U_i^j$ of $x_i^j$ and $V_i^j$ of $y_i^j$, there exists $n\in\mathbb{N}$ with $f_1^n(U_i^j)\cap V_i^j\ne\emptyset$, for all $i\in\{1, 2, \ldots, r\}$ and $j\in\{1, 2, \ldots, k\}$. Let  $z_i^j\in f_1^n(U_i^j)$ and $z_i^j\in V_i^j$, that is, $f_1^{-n}(z_i^j)\in U_i^j$ and $z_i^j\in V_i^j$ and hence using \eqref{M1}, we get $\rho_j = (\sum_{i=1}^r\delta_{z_i^j})/r\in\mathcal{V}_j$. Now, since $\widetilde{f}_1^{-n}(\rho_j)(U_i^j) = \rho_j(f_1^n(U_i^j)) = (\sum_{i=1}^r\delta_{z_i^j} f_1^n(U_i^j))/r$ and $z_i^j\in f_1^n(U_i^j)$, therefore  $\widetilde{f}_1^{-n}(\rho_j) = (\sum_{i=1}^r\delta_{f_1^{-n}(z_i^j)})/r\in\mathcal{U}_j$. Thus, $\widetilde{f}_1^n(\mathcal{U}_j)\cap \mathcal{V}_j\ne\emptyset$, for every $j\in\{ 1, 2, \ldots, k\}$, which gives that $(\mathcal{M}(X), \widetilde{f}_{1,\infty})$ is weakly mixing of all orders.

Conversely, for any $k\in\mathbb{N}$, let $U_1$, \ldots, $U_k$; $V_1$, \ldots, $V_k$ be any collection of non-empty open subsets of $X$. Let $\mathcal{U}_i = \{\mu\in\mathcal{M}(X) : \mu(U_i) > 2/5\}$ and $\mathcal{V}_i = \{\mu\in\mathcal{M}(X) : \mu(V_i) > 2/5\}$, for each $i\in\{ 1, 2, \ldots, k\}$, then $\mathcal{U}_i$, $\mathcal{V}_i$ are non-empty open subsets of $\mathcal{M}(X)$. Now, as $(\mathcal{M}(X), \widetilde{f}_{1,\infty})$ is weakly mixing of all orders, so there exists $n\in\mathbb{N}$ such that $\widetilde{f}_1^n(\mathcal{U}_i)\cap \mathcal{V}_i\ne\emptyset$, for each $i = 1, 2, \ldots, k$. Therefore, there exists $\nu_i\in \mathcal{U}_i$ with $\widetilde{f}_1^n(\nu_i)\in \mathcal{V}_i$ and hence $\widetilde{f}_1^n(\nu_i)(V_i) = \nu(f_1^{-n}(V_i))>2/5$, for every $i\in\{1, 2, \ldots, k\}$. Also, $\mu(U_i)>2/5$ implies that $f_1^n(U_i)\cap V_i\ne\emptyset$, for each $i\in\{ 1, 2, \ldots, k\}$. Thus, $(X, f_{1,\infty})$ is weakly mixing of all orders.
\end{proof}
\begin{cor} If $(\mathcal{M}(X), \widetilde{f}_{1,\infty})$ is weakly  mixing of order $m$, for any $m\geq 2$, then so is $(X, f_{1,\infty})$.
\end{cor}
For a closed unit interval $I$, by  \cite[Theorem 11]{1} and Theorem \ref{3.5}, we have the following two results. 
\begin{cor} If the non-autonomous system $(I, f_{1,\infty})$ is weakly mixing of order $3$, then $(\mathcal{M}(I), \widetilde{f}_{1,\infty})$ is weakly mixing of all orders.
\end{cor} 
Consequently, we have the following result.
\begin{cor} The non-autonomous system $(I, f_{1,\infty})$ is weakly mixing of order $3$ if and only if $(\mathcal{M}(I), \widetilde{f}_{1,\infty})$ is weakly mixing of order $3$.
\end{cor}
Now  we prove the main result of this section.
\begin{thm}\label{t3.5} For a commutative non-autonomous system $(X, f_{1,\infty})$, the following are equivalent:
\begin{enumerate}
\item[$(1)$]
$(X, f_{1,\infty})$ is weakly mixing.
\item[$(2)$]
$(\mathcal{M}(X), \widetilde{f}_{1,\infty})$ is weakly mixing.
\item[$(3)$]
$(\mathcal{M}(X), \widetilde{f}_{1,\infty})$ is topologically transitive.
\end{enumerate}
\end{thm}
\begin{proof} (1)$\Rightarrow$ (2). Since $(X, f_{1,\infty})$ is weakly mixing and $f_{1,\infty}$ is commutative, therefore by \cite[Lemma 3]{17}, $f_{1,\infty}$ is weakly mixing of all orders and hence by Theorem \ref{3.5}, we get that $(\mathcal{M}(X), \widetilde{f}_{1,\infty})$ is weakly mixing of all orders and hence it is weakly mixing. Clearly, $(2)\Rightarrow (3)$. 

Finally, we prove that $(3)\Rightarrow (1)$. Let $(\mathcal{M}(X), \widetilde{f}_{1,\infty})$ be topologically transitive, then by Theorem \ref{3.4}, $(X, f_{1,\infty})$ satisfies Banks's condition. We show that for a commutative family $f_{1,\infty}$ Banks's condition implies weakly mixing. Let $U_1$, $U_2$, $V_1$ and $V_2$ be arbitrary non-empty open subsets of $X$. Since $(X, f_{1,\infty})$ satisfies Banks's condition, therefore there exists $n\in\mathbb{N}$ such that $f_1^n(U_1)\cap U_2\ne\emptyset$ and $f_1^n(U_1)\cap V_2\ne\emptyset$. Consequently, $U = U_1\cap f_1^{-n}(U_2)$ and $V = f_1^{-n}(V_2)$ are non-empty open subsets of $X$. Applying Banks's condition to $U$, $V$ and $V_1$, we get a natural number $k$ satisfying $f_1^k(U)\cap V_1\ne\emptyset$ and $f_1^k(U)\cap V\ne\emptyset$. Hence, there exists $x\in U$ such that $f_1^k(x)\in f_1^{-n}(V_2)$, which implies $f_1^n\circ f_1^k(x)\in V_2$. Now, as the family $f_{1,\infty}$ is commutative, so $f_1^k\circ f_1^n(x)\in V_2$. Also, since $f_1^n(x)\in U_2$, therefore $f_1^k(U_2)\cap V_2\ne\emptyset$ and $U\subseteq U_1$ implies that $f_1^k(U_1)\cap V_1\ne\emptyset$. Thus, $(X, f_{1,\infty})$ is weakly mixing, which completes the proof.
\end{proof}
\begin{cor}\label{c3.5} Let $(X, f_{1,\infty})$ be a non-autonomous system  such that the family $f_{1,\infty}$ is commutative. If the induced system $(\mathcal{M}(X), \widetilde{f}_{1,\infty})$ is topologically transitive and has dense periodic points, then it is Devaney chaotic.
\end{cor}

Next, we discuss weakly mixing of the systems induced on hyperspaces. By \cite[Proposition 3.5, 3.10]{16} and proceeding as in the Theorem \ref{t3.5}, we get the following famous Banks's theorem for commutative non-autonomous systems.
\begin{thm}\label{3.2} For a commutative non-autonomous system $(X, f_{1,\infty})$, the following are equivalent:
\begin{enumerate}
\item[$(1)$]
$(X, f_{1,\infty})$ is weakly mixing.
\item[$(2)$]
$(\mathcal{K}(X), \overline{f}_{1,\infty})$ is weakly mixing.
\item[$(3)$]
$(\mathcal{K}(X), \overline{f}_{1,\infty})$ is topologically transitive.
\end{enumerate}
\end{thm}
\begin{rmk} Note that  if the family $f_{1,\infty}$ is not commutative, then above result is not true as justified by  \cite[Example 3.2, 3.4]{16}.
\end{rmk} 
By Theorem \ref{3.2} and slight modification in the proof of \cite[Theorem 1.1]{9} for autonomous systems, we get the following result.
\begin{cor}\label{c3.6} If the family $f_{1,\infty}$ is commutative and $(X, f_{1,\infty})$ is the corresponding non-autonomous system, then the following are equivalent:
\begin{enumerate}
\item[$(1)$]
$(X, f_{1,\infty})$ is weakly mixing with dense small periodic set.
\item[$(2)$]
$(\mathcal{K}(X), \overline{f}_{1,\infty})$ is weakly mixing with dense small periodic set.
\item[$(3)$]
$(\mathcal{K}(X), \overline{f}_{1,\infty})$ is Devaney chaotic. 
\end{enumerate}
\end{cor}
To justify  that the above result is not true if the family $f_{1,\infty}$ is not commutative we recall the following example from \cite{16}.
\begin{exm}\label{e3.1} Let $I$ be the closed unit interval and $F$ be the set of elements $(a, b, c, d)$ such that $a, b, c, d \in \mathbb{Q} \cap (0,1), a < b; \  a \ne c, b \ne d \ \text{and} \ c \ne d$.  To every element $(a, b, c, d)$ in $F$, a  homeomorphism $f: I \to I$ is assigned in two cases. If $c < d$, then $f$ is a function whose graph is determined by the segments $[(0,0), (a,c)], [(a,c), (b,d)]$ and $[(b,d),(1,1)]$. If $c > d$, then $f$ is a function whose graph is determined by the segments $[(0,1), (a,c)], [(a,c), (b,d)]$ and $[(b,d),(1,0)]$. In both the cases $f(a) = c$ and $f(b) = d$. Let $\{f_n : n\in\mathbb{N}\}$ be an enumeration of functions induced by the elements of $F$. Consider the non-autonomous system $(I, f_{1,\infty})$, where $f_{1,\infty} = \{f_1, f_1^{-1}, f_2, f_2^{-1}, \ldots, f_n, f_n^{-1}, f_{n+1}, f_{n+1}^{-1}, \ldots\}$. Now, as $f_1^{2n-1} = f_n$, for each $n\in\mathbb{N}$, so $f_{1,\infty}$ is weakly mixing. Also, $f_1^{2k}(x) = x$, for every $k\in\mathbb{N}$ and each $x\in I$ implying that every point is periodic point and hence the system $(X, f_{1,\infty})$ has dense small periodic set. But  $(\mathcal{K}(X), \overline{f}_{1,\infty})$ is not topologically transitive as discussed in \cite[Example 3.2]{16}, so it cannot be Devaney chaotic.
\end{exm}

\section{Weakly Mixing and Related Properties}
In this section, we give counter examples for the results which are true in autonomous systems but not in non-autonomous systems. For a commutative finitely generated non-autonomous system, it is shown that every weakly mixing system is thickly sensitive. It is proved that on commutative periodic non-autonomous systems weakly mixing implies Devaney chaos. Throughout this section,  $f_{1,\infty}$ denote the surjective family, that is, each $f_i$ is surjective.

For autonomous systems, it is known that if $f$ is weakly mixing then $f^k$ is also weakly mixing, for every $k\geq 2$. We claim that this is not true for non-autonomous system in general as justified by the following example.
\begin{exm}\label{e4.1} Let $\Sigma_2 =\{0, 1\}^{\mathbb{Z}}$ be the collection of two-sided sequences of $0$ and $1$, endowed with product topology. Define $\sigma: \Sigma_2 \to \Sigma_2$ by $\sigma(x) = ( \ldots, x_{-2}, x_{-1}, x_0, \fbox{$x_1$}, x_2$, \ldots), where $x = ( \ldots, x_{-2}, x_{-1}, $ \fbox{$x_0$}, $ x_1, x_2 \ldots) \in \Sigma_2$, then $\sigma$ is a homeomorphism and is called the \emph{shift map} on $\Sigma_2$. For any $k\geq 2$, we consider two cases, when $k$ is even, then we consider the non-autonomous system  $(\Sigma_2, f_{1,\infty})$, where $f_{1,\infty} = \{\sigma, \sigma^{-1}, \sigma^2, \sigma^{-2}, \sigma^3, \sigma^{-3}, \ldots\}$ and when $k$ is odd, then we consider the system $(\Sigma_2, g_{1,\infty})$, with $g_{1,\infty} = \{\sigma, \sigma^{-1}$, $\underbrace{i,\ldots, i}_{\text{$(k-2)$-times}} \sigma^2, \sigma^{-2}, \underbrace{i \ldots, i}_{\text{$(k-2)$-times}}$ $\sigma^3, \sigma^{-3},\ldots\}$,
where $i : \Sigma_2\to\Sigma_2$ is the identity map.
Note that $f_1^{2l-1}= \sigma^l$, for any $l\in\mathbb{N}$; $g_1^{(r-1)k+1} = \sigma^r$, for any $r\in\mathbb{N}$ and hence both  $(X, f_{1,\infty})$ and $(\Sigma_2, g_{1,\infty})$ are weakly mixing using the fact that $\sigma$ is topological mixing. But as $f_1^{kl} = i$, for $k$ even and each $l\in\mathbb{N}$  and $g_1^{kr} = i$, for $k$ odd and each $r\in\mathbb{N}$, so $f_{1,\infty}^{[k]}$, for $k$ even and $g_{1,\infty}^{[k]}$, for $k$ odd are not topologically transitive and hence cannot be totally transitive or weakly mixing. 
 \end{exm}
\begin{rmk} The above example also shows that weakly mixing need not imply totally transitive for non-autonomous systems  in contrast to the case of autonomous systems. Note that in the above example $(\Sigma_2, f_{1,\infty})$ or $(\Sigma_2, g_{1,\infty})$ is not topological mixing and both $f_{1,\infty}$, $g_{1,\infty}$ are  commutative and surjective.
\end{rmk}
\begin{thm}If the non-autonomous system $(X, f_{1,\infty})$ is topological mixing, then $(X, f_{1,\infty}^{[k]})$ is weakly mixing, for every $k\geq 2$.
\end{thm}
\begin{proof}
Let $U_i$, $V_i$ be any  non-empty open subsets of $X$, for $i =1, 2$. By topological mixing of $f_{1,\infty}$, there exist $n_1$, $n_2\in\mathbb{N}$ such that $f_1^n(U_1)\cap V_1\ne\emptyset$, for all $n\geq n_1$ and $f_1^n(U_2)\cap V_2\ne\emptyset$, for all $n\geq n_2$. Taking $N = \text{max}\{n_1, n_2\}$, we get that $f_1^n(U_i)\cap V_i\ne\emptyset$, for all $n\geq N$ and $i =1, 2$. Now, for any $k\in\mathbb{N}$, by Archimedean property of $\mathbb{R}$, there exists $r\in\mathbb{N}$ such that $rk > N$. Therefore, $f_1^{rk}(U_i)\cap V_i\ne\emptyset$, that is, there exists $r\in\mathbb{N}$ with $(f^k_{k(r-1)+1} \circ \cdots \circ f_1^k)(U_i)\cap V_i\ne\emptyset$, for $i=1, 2$ and each $k\in\mathbb{N}$. Hence, $(X, f_{1,\infty}^{[k]})$ is weakly mixing, for every $k\geq 2$.
\end{proof}
\begin{thm}Let $(X, f_{1,\infty})$ be an $m$-period commutative non-autonomous system, then $(X, f_{1,\infty})$ is weakly mixing if and only if $f_{1, \infty}^{[k]}$ is weakly mixing, for every $k\geq 1$.
\end{thm}
\begin{proof}
If $(X, f_{1,\infty})$ is weakly mixing, then by \cite[Lemma 10]{17}, the autonomous system $(X, g)$ is weakly mixing, where $g = f_m\circ\cdots\circ f_1$. Now, for autonomous system $(X, g)$, we have $g$ is weakly mixing if and only if $g^k$ is so, for any $k\in\mathbb{N}$. Therefore, for any  non-empty open subsets $U_i$, $V_i$  of $X$, for $i =1, 2$, there exists a natural number $r$ such that $g^{kr}(U_i)\cap V_i\ne\emptyset$, that is, $(f_1^m)^{rk}(U_i)\cap V_i\ne\emptyset$, for $i =1, 2$. Now, as the system $(X, f_{1,\infty})$ is $m$-periodic, so $(f_1^m)^{rk} = f_1^{mrk}$ and hence for $s = mr$, we get that $f_1^{sk}(U_i)\cap V_i\ne\emptyset$, that is, $(f^k_{k(s-1)+1} \circ \cdots \circ f_1^k)(U_i)\cap V_i\ne\emptyset$, for $i =1, 2$. Thus, $f_{1, \infty}^{[k]}$ is weakly mixing, for every $k\geq 1$.
\end{proof}
\begin{cor} For an $m$-periodic commutative non-autonomous system  $(X, f_{1,\infty})$, weakly mixing implies  total transitivity. 
\end{cor}
For autonomous systems, it is known that $(X,f)$ is weakly mixing if and only if $N_f(U, V)$ is thick, for every pair of non-empty open sets $U$ and $V$ of $X.$ This equivalence is not true in non-autonomous systems as justified by the example given below.
\begin{exm} Consider the non-autonomous system $(\Sigma_2, f_{1,\infty})$ from Example \ref{e4.1}, that is, $f_{1,\infty} = \{\sigma, \sigma^{-1}, \sigma^2, \sigma^{-2}, \sigma^3, \sigma^{-3}, \ldots\}$. Let $U_1$, $U_2$; $V_1$, $V_2$ be   non-empty open subsets of $\Sigma_2$. Since $\sigma$ is topologically mixing, therefore there exists $k\in\mathbb{N}$ such that $\sigma^n(U_i) \cap V_i\ne\emptyset$, for all $n\geq k$ and each $i=1, 2$. Now, $f_1^{2k-1} = \sigma^k$, which implies that $f_1^{2k-1}(U_i) \cap V_i\ne\emptyset$, for  each $i=1, 2$ and hence $f_{1,\infty}$ is  weakly mixing. Note that for each $i=1, 2$, we have $N_{f_{1,\infty}}(U_i, V_i) = \{2k-1, 2k+1, 2k+3, \ldots\}$, which is not thick.

Note that since the family $f_{1,\infty}$ is commutative and $(\Sigma_2, f_{1,\infty})$ is weakly mixing, therefore by Theorem \ref{t3.5}, $(\mathcal{M}(\Sigma_2), \widetilde{f}_{1,\infty})$ is weakly mixing. Also, as $f_1^{2k}(x) = x$, for every $k\in\mathbb{N}$, so $(\Sigma_2, f_{1,\infty})$ has dense  periodic points and hence $(\mathcal{K}(\Sigma_2), \overline{f}_{1,\infty})$ is weakly mixing and Devaney chaotic.
\end{exm}
In \cite{14}, we gave an example to show that weakly mixing need not imply thick sensitivity in non-autonomous systems. Now, we give the sufficient condition under which weakly mixing implies thick sensitivity in non-autonomous systems.

\begin{thm}Let $f_{1,\infty}$ be a finitely generated commutative family. If $(X, f_{1,\infty})$ is weakly mixing, then it is thickly sensitive.
\end{thm}
\begin{proof}
Let $U$ be any non-empty open subset of $X$ and $0<\epsilon < \text{diam}(X)/3$. Since $(X, f_{1,\infty})$ is weakly mixing and $f_{1,\infty}$ is commutative, therefore by Theorem \ref{3.2}, $(\mathcal{K}(X)$, $\overline{f}_{1,\infty})$ is topologically transitive. Also, each $f_i$ being surjective,  we have $(\mathcal{K}(X), \overline{f}_{1,\infty})$ cannot be minimal. As $\mathcal{K}(X)$ is compact, so  $\overline{f}_{1,\infty}$ has dense transitive points in $\mathcal{K}(X)$ and hence there exists a transitive point $S$ of $\mathcal{K}(X)$ such that $S\subseteq \mathcal{U} =\langle U\rangle$. We first claim that $N(S, B_{\mathcal{H}}(X, \epsilon)) := \{n\in\mathbb{N} : \overline{f}_1^n(K)\in B_{\mathcal{H}}(X, \epsilon)\}$ is thick. If $(X, f_{1,\infty})$ is finitely generated, then so is $(\mathcal{K}(X), \overline{f}_{1,\infty})$ and since $\mathcal{K}(X)$ is compact also, therefore there exists $\delta>0$ such that for any $A$, $B\in\mathcal{K}(X)$, for each $j\in\{1, 2, \ldots, k\}$ and for all $t\geq 1$, we get that
\begin{equation}\label{W1}
\mathcal{H}(A, B)<\delta \ \Rightarrow \ \mathcal{H}(\overline{f}_t^j(A), \overline{f}_t^j(B))<\epsilon.
\end{equation}
Now, $S$ being a transitive point, we have an existence of a natural number $m$ with $\overline{f}_1^m(S)\in B_{\mathcal{H}}(X, \delta)$. Thus, by \eqref{W1}, we get that $\mathcal{H}(\overline{f}_{m+1}^j(\overline{f}_1^m(S), \overline{f}_{m+1}^j(X))<\epsilon$, that is, $\mathcal{H}(\overline{f}_1^{m+j}(S), X)<\epsilon$, for $j =1, \ldots, k$ implying that $N(S, B_{\mathcal{H}}(X, \epsilon))$ is thick. Next, we show that $N(S, B(X, \epsilon))\subseteq N_{f_{1,\infty}}(U, \epsilon)$ which will give that $N_{f_{1,\infty}}(U, \epsilon)$ is thick. Let $l\in N(S, B(X, \epsilon))$, then for $s_1$, $s_2\in S\subseteq U$ and for any $x_1$, $x_2\in X$, we have $d(f_1^l(s_i), x_i)<\epsilon$, for $i= 1, 2$. By triangle inequality, we have
\begin{align*}3\epsilon\leq d(x_1, x_2) & \leq d(x_1, f_1^l(s_1)) + d(f_1^l(s_1), f_1^l(s_2)) + d(f_1^l(s_2), x_2) \\ & < 2\epsilon + d(f_1^l(s_1), f_1^l(s_2)),\end{align*} which implies that $d(f_1^l(s_1), f_1^l(s_2))>\epsilon$ and hence $l\in N_{f_{1,\infty}}(U, \epsilon)$ which gives that $N(S$, $B(X, \epsilon))\subseteq N_{f_{1,\infty}}(U, \epsilon)$.
\end{proof}
\begin{exm} Consider a $3$-periodic non-autonomous system $(X, f_{1,\infty})$, where $f_1 = \sigma$, $f_2 = \sigma^{-2}$, $f_3 = \sigma^2$ and $X = \Sigma_2$, that is, 
$f_{1,\infty} = \{\sigma, \sigma^{-2}, \sigma^2, \sigma, \sigma^{-2}, \sigma^2, \ldots\}$, where $\sigma$ is the shift map as defined in Example \ref{e4.1}. Then $(\Sigma_2, f_{1,\infty})$ is finitely generated and commutative. Also, since $f_3\circ f_2\circ f_1 = \sigma$ and $\sigma$ is weakly mixing, therefore $(\Sigma_2, f_{1,\infty})$ is also weakly mixing. Hence, we get that $(\Sigma_2, f_{1,\infty})$ is thickly sensitive and totally transitive.
\end{exm} 
\begin{lem}\label{l4.5} If $(X, f_{1,\infty})$ is a non-autonomous system such that each $f_i$ is  an isometry with $X$ infinite, then $(X, f_{1,\infty})$ does not satisfy Banks's condition.
\end{lem}
\begin{proof}
Let $x$, $y$, $z\in X$ be three distinct points and $\delta := \min$ $\{d(x, y), d(y, z), d(z, x)\}/4$. Let $U$, $V_1$, $V_2$ be open balls of radius $\delta$ around $x$, $y$ and $z$, respectively. Since diam$(U)\leq 2\delta$ and each $f_i$ is an isometry, therefore diam$(f_1^n(U))\leq 2\delta$, for each $n\in\mathbb{N}$. Assume that $f_{1,\infty}$ satisfies Banks's condition, then there exists $k\in\mathbb{N}$ with $f_1^k(U)\cap V_1\ne\emptyset$ and $f_1^k(U)\cap V_2\ne\emptyset$. Therefore, there exist $p$ and $q$ such that $p$, $q\in f_1^n(U)$ and $d(y, p)<\delta$, $d(z, q)<\delta$. Now, as diam$(f_1^n(U))\leq 2\delta$, so $d(p, q)\leq 2\delta$. Also, by triangle inequality, we have 
\begin{align*}
4\delta \leq d(y, z) & \leq d(y, p) + d(p, q) + d(q, z) < 2\delta + d(p, q),
\end{align*}
implying that $d(p, q) > 2\delta$, which is a contradiction. Thus, $(X, f_{1,\infty})$ cannot satisfy Banks's condition.
\end{proof}
For autonomous dynamical systems it is known that if the system is topologically transitive and has dense prime period periodic points then it is weakly mixing \cite{13}. We assert that this result is not true for non-autonomous system and give the following example to support our claim.
\begin{exm}\label{e4.4} Let $S^1$ be the unit circle and $\alpha$ be an irrational number. Let $f$, $g : S^1\to S^1$ be defined as $f(e^{i\theta}) = e^{i(\theta + 2\pi\alpha)}$ and $g(e^{i\theta}) = e^{i(\theta - 2\pi\alpha)}$, then $f$ and $g$ are irrational rotations on $S^1$ which are minimal and hence topologically transitive. Consider the non-autonomous system $(S^1, f_{1,\infty})$, where $f_{2n-1}(e^{i\theta}) = f^n(e^{i\theta}) $ and $f_{2n}(e^{i\theta}) = g^n(e^{i\theta})$, that is, \[f_{1,\infty} = \{e^{i(\theta + 2\pi\alpha)}, e^{i(\theta - 2\pi\alpha)}, e^{i(\theta + 4\pi\alpha)}, e^{i(\theta - 4\pi\alpha)}, \ldots, e^{i(\theta + 2n\pi\alpha)}, e^{i(\theta - 2n\pi\alpha)}, \ldots\}.\] Note that each $f_i$ is an isometry and hence by Lemma \ref{l4.5}, $(S^1, f_{1,\infty})$ is not weakly mixing. Now, since $f_1^{2k}(e^{i\theta}) = e^{i\theta}$, for every $k\in\mathbb{N}$, therefore each point of $S^1$ is periodic point of prime period $2$. Thus, $(S^1, f_{1,\infty})$ has dense prime period periodic points. Also, $f_1^{2m-1} = f^m$, for every $m\in\mathbb{N}$ and $f$ is topologically transitive, therefore $(S^1, f_{1,\infty})$ is also topologically transitive. 

Now, let   $e^{i\theta}$, $e^{i\phi}\in S^1$ be any two distinct points and $\mu = \delta_{e^{i\theta}}$, $\nu = \delta_{e^{i\phi}}\in\mathcal{M}(S^1)$.  If $U$ and $V$ are sufficiently small neighborhoods of $\mu$ and $\nu$, respectively, then $\widetilde{f}^n(U)\cap V = \emptyset$, for each $n\in\mathbb{N}$ \cite{3}. For $l = 2n-1$, we have $f_1^{l} = f^n$ and for $l = 2n$, we have $f_1^l = i_{S^1}$, where $i_{S^1}$ is the identity map on $S^1$, and hence $\widetilde{f}_1^n(U)\cap V = \emptyset$, for every $n\in\mathbb{N}$. Thus, $(\mathcal{M}(S^1), \widetilde{f}_{1,\infty})$ cannot be topologically transitive. 
\end{exm}
\begin{rmk} We have the following conclusions.
\begin{enumerate}
\item[(1)]
Note that in the above Example \ref{e4.4}, the system is also not totally transitive. Therefore, if the non-autonomous system is topologically transitive and has dense prime period points then it is not even totally transitive.
\item[(2)]
Above example also shows that the converse of the statement in Corollary \ref{c3.1} need not be true for non-autonomous discrete systems.
\item[(3)]
For an autonomous system, it is known that every periodic orbit is finite or if the system is minimal then it has no periodic point. However, this need not be true for non-autonomous systems. In the above Example \ref{e4.4} every point is a periodic point and since $\mathcal{O}_{f_{1,\infty}}(e^{i\theta}) = \mathcal{O}_f(e^{i\theta})$, for every $x\in S^1$, therefore $(S^1, f_{1,\infty})$ is minimal as $(S^1, f)$ is minimal. Thus, we have a non-autonomous system which is minimal and has each point as periodic point.
\item[(4)]
The above example also shows that  the non-autonomous system which is topologically transitive having dense set of periodic points need not imply sensitivity even if the family $f_{1,\infty}$ is commutative. Whereas for the commutative induced non-autonomous systems, we have proved that topological transitivity and dense periodicity imply sensitivity, see Corollary \ref{c3.5} and Corollary \ref{c3.6}.
\end{enumerate}
\end{rmk}

We end this section by giving  sufficient conditions  under which a non-autonomous system becomes Devaney chaotic.
\begin{thm} Let $(J, f_{1,\infty})$ be a commutative $m$-periodic non-autonomous system, where $J$ is an interval. If the system $(J, f_{1,\infty})$ is weakly mixing, then it is Devaney chaotic.
\end{thm}
\begin{proof}
Let $(J, f_{1,\infty})$ be weakly mixing, then by \cite[Lemma 10]{17}, the autonomous system $(J, g)$ is weakly mixing, where $g = f_m\circ\cdots\circ f_1$. Now, for autonomous systems we have topological transitivity on intervals imply Devaney chaos, therefore $(J, g)$ has dense set of periodic points and is sensitive. Hence, by \cite[Lemma 1]{17}, we get that $(J, f_{1,\infty})$ has dense set of periodic points. Let $x\in J$ be arbitrary and $U$ be an open neighborhood of $x$. Since $(J, g)$ is sensitive, therefore there exists $\delta>0$ such that for above $x$, there exist  $y\in U$ and $n\in\mathbb{N}$ such that $d(g^n(x), g^n(y))>\delta$, that is, $d((f_1^m)^n(x), (f_1^m)^n(y))>\delta$. Now, since the system $(J, f_{1,\infty})$ is $m$-periodic, therefore $f_1^{mn} = (f_1^m)^n$. Hence, we get that $d(f_1^{nm}(x), f_1^{nm}(y))>\delta$. Thus, there exists $r = nm\in\mathbb{N}$ such that $d(f_1^r(x), f_1^r(y))>\delta$ implying that $(J, f_{1,\infty})$ is sensitive. Therefore, $(J, f_{1,\infty})$ is Devaney chaotic.
\end{proof}
By \cite[Lemma 1]{17}, \cite[Lemma 2]{17} and sensitivity from above theorem, we get the following results.
\begin{cor} Let $(X, f_{1,\infty})$ be an $m$-periodic non-autonomous system and $g = f_m\circ\cdots\circ f_1$. If the autonomous system $(X, g)$ is topologically transitive and has dense set of periodic points, then $(X, f_{1,\infty})$ is  Devaney chaotic.
\end{cor}
\begin{cor} Let $(J, f_{1,\infty})$ be an $m$-periodic non-autonomous system and $g = f_m\circ\cdots\circ f_1$, where $J$ is any interval. If the autonomous system $(J, g)$ is topologically transitive, then $(J, f_{1,\infty})$ is  Devaney chaotic.
\end{cor}
\section{Sensitivity on the unit interval [0, 1]}
In this section, we first give examples to justify that for non-autonomous systems in general sensitivity need not be equivalent to cofinite sensitivity and other stronger forms of sensitivity. Sufficient condition is given under which sensitivity is equivalent to some stronger forms of sensitivity in non-autonomous systems.

For autonomous dynamical systems, Moothathu proved that on closed unit interval $I$ sensitivity is equivalent to the strongest form of sensitivity, namely cofinite sensitivity \cite{12}. Hence, all other stronger forms of sensitivities are equivalent to sensitivity on $I$ in autonomous systems. We assert that for non-autonomous system this need not be true. 

The following example shows that sensitivity is not equivalent to cofinite and thick sensitivity for non-autonomous systems in general.
\begin{exm} Let $I$ be the closed unit interval and $\{f_n : n\in\mathbb{N}\}$ be an enumeration of functions induced by the elements of $F$ as considered in Example \ref{e3.1}. Let  $(I, f_{1,\infty})$ be the non-autonomous, where $f_{1,\infty} = \{f_1, f_1^{-1}, i, f_2, f_2^{-1}, i, \ldots, f_n, f_n^{-1}, i,  f_{n+1}, f_{n+1}^{-1}, i \ldots\}$, here $i: I\to I$ is the identity map. Since $f_1^{3r-2} = f_r$, for each $r\in\mathbb{N}$, therefore as before $(I, f_{1,\infty})$ is weakly mixing and hence sensitive dependence on initial conditions. Note that as $f_1^{3k}(x) = x$, for every $x\in I$, so $(I, f_{1,\infty})$ cannot be cofinitely sensitive or thickly sensitive.
\end{exm}
The following example shows that sensitivity is not equivalent to syndetic and ergodic sensitivity for non-autonomous systems in general. 
\begin{exm}\label{se} Let $I$ be the closed unit interval and $\{f_n : n\in\mathbb{N}\}$ be an enumeration of functions induced by the elements of $F$ as considered in Example \ref{e3.1}. Let  $(I, f_{1,\infty})$ be the non-autonomous, where 
\[f_{1,\infty} = \{f_1, f_1^{-1}, \underbrace{i,\ldots, i}_{\text{10-times}}, f_2, f_2^{-1}, \underbrace{i,\ldots, i}_{\text{$10^2$-times}}, f_3, f_3^{-1},\ldots, \underbrace{i,\ldots, i}_{\text{$10^{n-1}$-times}}, f_n, f_n^{-1}, \dots\},\] where $i: I\to I$ is the identity map. Since $f_1^k = f_n$, for every $n\in\mathbb{N}$, where $k = 10 + 10^2 +\cdots + 10^{n-1}+ 2n-1$, therefore $(I, f_{1,\infty})$ is weakly mixing and hence sensitive. Thus, for any $x\in I$ and any open neighborhood $U$ of $x$, there is $\delta>0$ such that $N_{f_{1,\infty}}(U, \delta)\ne\emptyset$ and is infinite. Let $N_{f_{1,\infty}}(U, \delta) = \{n_1, n_2, n_3, n_4, \ldots\}$, then $n_2-n_1< n_3-n_2 < n_4 - n_3 < \cdots$, therefore $N_{f_{1,\infty}}(U, \delta)$ has large unbounded gaps and hence $(I, f_{1,\infty})$ cannot be syndetically sensitive. Also, note that $d(N_{f_{1,\infty}}(U, \delta)) = 0$ implying that $(I, f_{1,\infty})$ is not ergodically sensitive.
\end{exm}
We now give  sufficient conditions under which some of the stronger forms of sensitivity become equivalent to sensitivity on $I$ in non-autonomous systems.
\begin{thm}\label{4.3} Let $(I, f_{1,\infty})$ be an $m$-periodic non-autonomous system. If a non-autonomous system $(I, f_{1,\infty})$ is sensitive, then it is syndetically sensitive, ergodic sensitive, collectively sensitive and multi-sensitive.
\end{thm}
\begin{proof} We first claim that for any compact metric space $X$, if $(X, f_{1,\infty})$ is sensitive, then the corresponding autonomous system $(X, g)$ is sensitive, where $g = f_m\circ\cdots\circ f_1$. Let $(X, f_{1,\infty})$ be sensitive with $\delta>0$ as a constant of sensitivity. Since each $f_j$ is uniformly continuous, therefore for each positive integer $0\leq l\leq m+2$, $f_j^l$ is uniformly continuous. Now, as $f_{1,\infty}$ is $m$-periodic, so by uniform continuity, for the above  $\delta>0$, there exists $\epsilon>0$ such that for  each $0\leq l\leq m+2$ and any $x$, $y\in X$, $j>0$, we have
 \begin{equation}\label{3}
 d(x, y)<\epsilon \ \Rightarrow \  d(f_j^l(x), f_j^l(y))<\delta. 
 \end{equation}
 By sensitivity of $(X, f_{1,\infty})$, we have for any $x\in X$ and any open neighborhood $U$ of $x$ and in particular for $\xi$-neighborhood  $(0<\xi<\epsilon)$ of $x$ there exist $y\in B(x, \xi)$ and $n_{\xi}\in\mathbb{N}$ such that $d(f_1^{n_{\xi}}(x), f_1^{n_{\xi}}(y))>\delta$. Now as $\xi<\epsilon$, so using (\ref{3}), we have $n_{\xi}>m+1>m$ and using division algorithm, we get a positive integer $r$ such that $n_{\xi} = rm + q$, $0\leq q\leq m-1$. Thus \begin{align*}d(f_1^{n_{\xi}}(x), f_1^{n_{\xi}}(y)) & = d(f_1^{rm+q}(x), f_1^{rm+q}(y)) \\ & = d(f_{rm+1}^q(f_1^{rm}(x)), f_{rm+1}^q(f_1^{rm}(y))).\end{align*} Since $q = (n_{\xi} -rm)\leq (m-1)<(m+2)$, therefore using (\ref{3}) and $d(f_{rm+1}^q(f_1^{rm}(x))$, $f_{rm+1}^q(f_1^{rm}(y)))$ $>\delta$, we get that $d(f_1^{rm}(x), f_1^{rm}(y))\geq \epsilon$. Taking $0<\eta<\epsilon$, we get $d(f_1^{rm}(x), f_1^{rm}(y))>\eta$, that is, $d(g^r(x), g^r(y))>\eta$, implying $(X, g)$ is sensitive. Hence, $(I, g)$ is sensitive.  Now, for autonomous systems sensitivity on $I$ implies cofinite sensitivity and hence $(I, g)$ is syndetically sensitive, ergodic sensitive, collective sensitive and multi-sensitive.
 
If $(I, g)$ is syndetically sensitive, then  for any $x\in I$ and any open neighborhood $U$ of $x$, there exist $\delta>0$ and $n_i\in\mathbb{N}$ with $N_{g}(U, \delta) = \{n_1, n_2, n_3, n_4, \ldots\}$ such that $n_{i+1}-n_{i}\leq M$, for all $i\geq 1$ and for some positive constant $M$. Since $(I, f_{1,\infty})$ is $m$-periodic, therefore $g^{n_i} = (f_1^m)^{n_i} = f_1^{mn_i}$ and hence $N_{f_{1,\infty}}(U, \delta) = \{k_1, k_2, k_3, k_4,\ldots\}$, where $k_i = mn_i$. Note that $(k_{i+1} - k_i)\leq M/m$, for all $i\geq 1$.  Thus, $N_{f_{1,\infty}}(U, \delta)$ has bounded gaps implying that $(I, f_{1,\infty})$ is syndetically sensitive.

Since syndetical sensitivity implies ergodic sensitivity, therefore $(I, f_{1,\infty})$ is also ergodic sensitive. We also show it holds directly. If $(I, g)$ is ergodically  sensitive, then for any $x\in I$ and any open neighborhood $U$ of $x$, there exists $\delta>0$ such that $N_{g}(U, \delta)\ne\emptyset$ with $\overline{d}(N_{g}(U, \delta))>0$. Since $(I, f_{1,\infty})$ is $m$-periodic, therefore $g^{n} = (f_1^m)^{n} = f_1^{mn}$, for every $n\in\mathbb{N}$ and hence $\overline{d}(N_{f_{1,\infty}}(U, \delta))>0$. Thus, the system $(I, f_{1,\infty})$ is ergodically sensitive.  

 Now, suppose that $(X, g)$ is collectively sensitive and let $x_1$, $x_2$, \ldots, $x_k\in X$ and $\epsilon>0$  is given. Since $(X, g)$ is collectively sensitive, therefore there exists $\delta>0$ such that corresponding to above $x_i$'s, there exist  $y_i\in X$, $i= 1, 2,\ldots, k$ with $d(x_i, y_i)<\epsilon$, $n\in\mathbb{N}$ and $i_0$ with $1\leq i_0\leq k$ such that $d(g^n(x_i), g^n(y_{i_0}))>\delta$ or $d(g^n(y_i), g^n(x_{i_0}))>\delta$, that is, $d((f_1^m)^n(x_i), (f_1^m)^n(y_{i_0}))>\delta$ or $d((f_1^m)^n(y_i), (f_1^m)^n(x_{i_0}))>\delta$, for $i = 1, 2, \ldots, k$. Now, since the system $(X, f_{1,\infty})$ is $m$-periodic, therefore $f_1^{mn} = (f_1^m)^n$. Hence, we get that $d(f_1^{nm}(x_i), f_1^{nm}(y_{i_0}))>\delta$ or $d(f_1^{nm}(y_i), f_1^{nm}(x_{i_0}))>\delta$. Thus, there exists $p = nm\in\mathbb{N}$ such that $d(f_1^p(x_i), f_1^p(y_{i_0}))>\delta$ or $d(f_1^p(y_i), f_1^p(x_{i_0}))>\delta$, for $i = 1, 2, \ldots, k$ implying that $(X, f_{1,\infty})$ is collectively sensitive. Also, if $(I, g)$ is multi-sensitive, then by \cite[Theorem 3.2]{14}, the corresponding non-autonomous system $(I, f_{1,\infty})$ is multi-sensitive.
\end{proof}
\begin{cor}\label{c4.3} For an $m$-periodic non-autonomous system $(I, f_{1,\infty})$, weakly mixing implies ergodic sensitivity and  syndetical sensitivity.
\end{cor}
\begin{rmk} Note that Corollary \ref{c4.3} is not true if the non-autonomous is not $m$-periodic as justified by the Example \ref{se}. 
\end{rmk}

\end{document}